\documentclass{amsart}
\usepackage{amsmath}
\usepackage{amssymb}
\usepackage{amsthm}
\usepackage{enumerate}
\usepackage[pdftex]{graphicx}
\usepackage{caption}
\theoremstyle{definition}
\newtheorem{definition}{Definition}[section]
\theoremstyle{plain}
\newtheorem{lemma}[definition]{Lemma}
\newtheorem{theorem}[definition]{Theorem}

\newtheorem{proposition}[definition]{Proposition}
\newtheorem{corollary}[definition]{Corollary}
\theoremstyle{remark}
\newtheorem{remark}[definition]{Remark}

\makeatletter
\@namedef{subjclassname@2020}{
  \textup{2020} Mathematics Subject Classification}
\makeatother

\newcommand{\mycl}{\operatorname{cl}}
\newcommand{\myint}{\operatorname{int}}
\newcommand{\mybd}{\operatorname{bd}}

\begin{document}

\title[Invariance of domain]{Invariance of domain in locally o-minimal structures}
\author[M. Fujita]{Masato Fujita}
\address{Department of Liberal Arts,
	Japan Coast Guard Academy,
	5-1 Wakaba-cho, Kure, Hiroshima 737-8512, Japan}
\email{fujita.masato.p34@kyoto-u.jp}

\begin{abstract}
	%We prove the following theorem in this paper:
	Definable continuous injective maps defined on definable open sets into the Euclidean spaces of the same dimension are open maps in definably complete locally o-minimal expansions of ordered groups.
\end{abstract}



\thanks{The author was supported by JSPS KAKENHI Grant Number JP25K07109.}

\maketitle

\section{Introduction}\label{sec:intro}

Our aim is to prove the following theorem:

\begin{theorem}[Invariance of domain]\label{thm:open_map}
	Let $\mathcal M=(M,<,+,0,\ldots)$ be a definably complete locally o-minimal expansion of an ordered group.
	Let $V$ be a definable open subset of $M^n$ and $f:V \to M^n$ be a definable continuous injective map.
	Then $f$ is an open map; that is, $f(U)$ is open whenever $U$ is an open subset of $V$.
\end{theorem}
%Note that a continuous injective open map is a homeomorphism onto its image.

We recall several previous works.
A classical invariance of domain and its proof are found in books on algebraic topology such as  \cite{Dold}.
The notion of homology is often used for the proof.
Woreheide proved the above theorem when the structure $\mathcal M$ is an o-minimal expansion of an ordered field using the theory of definable homology \cite{Woerheide}.
Johns proved that Woreheide's result holds under a relaxed condition in \cite{Johns} without using the notion of homology.
He assumed that $\mathcal M$ is an o-minimal structure and used the definable cell decomposition theorem.
Pierzcha{\l}a proved the same result independently of Johns in \cite{P}.
Dinh and Ph\d{a}m gave necessary and sufficient conditions for continuous maps definable in an o-minimal expansion of the real field to be open in \cite{DP}. 
They use Brouwer degree of a definable map in their proof, and the Brouwer degree of a definable map is unavailable in our setting. 
%Readers who are not familiar with o-minimality should consult a standard textbook \cite{vdD}.

Several notions similar to o-minimality have been proposed.
We treat definably complete local o-minimality \cite{Miller,TV} among them.
An expansion of a dense linear order without endpoints $\mathcal M=(M,<,\ldots)$ is \textit{definably complete} if any definable subset $X$ of $M$ has the supremum and  infimum in $M \cup \{\pm \infty\}$ \cite{Miller}.
We say that $\mathcal M$ is \textit{locally o-minimal} if, for every definable subset $X$ of $M$ and for every point $a\in M$, we can choose an open interval $I$ containing the point $a$ so that $X \cap I$ is a union of a finite set  and finitely many open intervals \cite{TV}.
We fix a definably complete locally o-minimal expansion of an ordered group $\mathcal M$ in this paper.

Definable cell decomposition theorem is unavailable in the setting of Theorem \ref{thm:open_map}.
The author and his collaborators \cite{Fuji2,FKK} found that the dimension function behaves very tamely in definably complete locally o-minimal structures similarly to tame behavior of dimension function in o-minimal structures \cite[Chapter 4]{vdD}.
We prove the theorem using this tame behavior of dimension function.

We introduce the terms and notations used in this paper. 
Throughout, the term ‘definable’ means ‘definable in the given structure with parameters’.
We assume that $M$ is equipped with the order topology induced from the linear order $<$ and the topology on $M^n$ is the product topology of the order topology on $M$.
An open interval is a set of the form $(b_1,b_2):=\{x \in M\;|\; b_1<x<b_2\}$ for some $b_1,b_2 \in M \cup \{\pm \infty\}$ and open box is the product of open intervals.
We define a closed interval similarly and it is denoted by $[b_1,b_2]$.
For a subset $A$ of $ M^n$, $\myint(A)$, $\mycl(A)$, $\partial(A):=\mycl(A) \setminus A$ and $\mybd(A):=\mycl(A) \setminus \myint(A)$ denote the interior, the closure, the frontier and the boundary of $A$, respectively.

\section{Proof of the theorem}\label{sec:proof}

We begin to prove Theorem \ref{thm:open_map}.
Let $X$ be a definable subset of $M^n$.
We first recall the definition of $\dim X$.
The dimension of $X$ is the maximal nonnegative integer $d$ such that $\pi(X)$ has a nonempty interior for some coordinate projection $\pi:M^n \rightarrow M^d$.
We consider that $M^0$ is a singleton equipped with the trivial topology.
We set $\dim(X)=-\infty$ when $X$ is an empty set.
Tame properties of dimension function in definably complete locally o-minimal structures are summarized in \cite[Proposition 2.8]{FKK}.
Recall that a definable set $X$ is called \textit{definably connected} if it is not decomposed into two disjoint nonempty definable open subsets of $X$.

We say that $X$ is \textit{definably path connected} if, for every two points $x,y$ in $X$, there exists a definable continuous map $\gamma:[a,b] \to X$ such that $\gamma(a)=x$ and $\gamma(b)=y$.
Let $x,y$ be two points in $X$.
We say that $y$ is \textit{definably reachable} from $x$ in $X$ if there exist a natural number $m$ and definable continuous maps $\gamma_i:[a_i,b_i] \to X$ from nonempty bounded closed intervals $[a_i,b_i]$ for $1 \leq i \leq m$ such that $\gamma_1(a_1)=x$, $\gamma_m(b_m)=y$ and $\gamma_i(b_i)=\gamma_{i+1}(a_{i+1})$ for $1 \leq i <m$.
Definable reachability is obviously an equivalence relation in $X$.
In addition, when the structure is an expansion of an ordered group like our cases, $X$ is definably path connected if and only if every point $y$ in $X$ is definably reachable from an arbitrary point $x$ in $X$.
The following lemma is obvious and we omit the proof.

\begin{lemma}\label{lem:reachable}
	%Consider a definably complete structure.
	A definable set $X$ is definably connected if every point $y$ in $X$ is definably reachable from an arbitrary point $x$ in $X$.
\end{lemma}
%\begin{proof}
%	Assume for contradiction that $X$ is not definably connected.
%	We can take two nonempty definable subsets $A$ and $B$ of $X$ such that $A \cap B=\emptyset$, $A \cup B=X$, and both $A$ and $B$ are open in $X$.
%	
%	Take points $x \in A$ and $y \in B$.
%	Since the point $y$ is definably reachable from the point $x$ in $X$ by the assumption, there exist a natural number $m$ and definable continuous maps $\gamma_i:[a_i,b_i] \to X$ from nonempty closed intervals $[a_i,b_i]$ for $1 \leq i \leq m$ such that $\gamma_1(a_1)=x$, $\gamma_m(b_m)=y$ and $\gamma_i(b_i)=\gamma_{i+1}(a_{i+1})$ for $1 \leq i <m$.
%	Set $A_i=\gamma_i^{-1}(A)$ and $B_i=\gamma_i^{-1}(B)$ for every $1 \leq i \leq m$.
%	We have $B_m \neq \emptyset$ because $\gamma_m(b_m)=y \in B$.
%	Set $k=\min\{1 \leq i \leq m\;|\; B_i \neq \emptyset\}$.
%	We have $A_k \neq \emptyset$.
%	In fact, if $A_k=\emptyset$, we have $\gamma_k(a_k) \in B$.
%	It implies that $\gamma_{k-1}(b_{k-1}) = \gamma_k(a_k) \in B$ and $B_{k-1} \neq \emptyset$, which contradicts the definition of $k$.
%	
%	We have shown that $A_k$ and $B_k$ are nonempty.
%	The closed interval $[a_k,b_k]$ is decomposed into two disjoint nonempty definable open subsets $A_k$ and $B_k$ of $[a_k,b_k]$.
%	This contradicts \cite[Corollary 1.5]{Miller}.
%\end{proof}

We next prove the following simple and important proposition:
\begin{proposition}\label{prop:connected}
	%Consider a definably complete locally o-minimal expansion of an ordered group $\mathcal M=(M,<,\ldots)$.
	Let $B$ be a nonempty open box in $M^n$ and $X$ be a definable closed subset of $B$ of dimension $<n-1$.
	Then the definable set $B \setminus X$ is definably connected.
\end{proposition}
\begin{proof}
	We show the following claim:
	The proposition follows from the claim by Lemma \ref{lem:reachable}.
	\medskip
	
	\textbf{Claim.}
	We take arbitrary two points $x=(x_1,\ldots, x_n),y=(y_1,\ldots, y_n) \in B \setminus X$.
	The point $y$ is definably reachable from $x$ in $B \setminus X$.
	\medskip
	
	We prove the claim by induction on $n$.
	The proposition is obvious for $n=1$ because $X$ is empty and $B$ is an open interval.
	We consider the case $n>1$.
	Set $X_t:=\{z \in M^{n-1}\;|\; (z,t) \in X\}$ for each $t \in M$.
	We first reduce to the case in which $\dim X_{x_n} < n-2$.
	Consider the definable subset $A$ of $M$ defined by $A:=\{t \in M\;|\; \dim X_t = n-2\}$.
	Since $\dim X \leq n-2$, we have $\dim A=0$ by \cite[Proposition 2.8(11)]{FKK}.
	We can find $x'_n \in M$ sufficiently close to $x_n$ such that $x'_n> x_n$ and $x'_n \notin A$ because $A$ has an empty interior.
	Since $U \setminus X$ is open and $x'_n$ is sufficiently close to $x_n$, the point $(x_1,\ldots, x_{n-1},t)$ belongs to $U \setminus X$ for every $x_n \leq t \leq x'_n$.
	The point $(x_1,\ldots, x_{n-1},x'_n)$ is reachable from $x=(x_1,\ldots, x_{n-1},x_n)$ in $U \setminus X$.
	We may assume that $\dim X_{x_n}<n-2$ by replacing $x$ with $(x_1,\ldots, x_{n-1},x'_n)$ if necessary.
	
	Let $\pi:M^n \to M^{n-1}$ be the projection forgetting the last coordinate. 
	We next reduce to the case in which $\pi(y) \not\in \pi(X)$.
	We have nothing to prove when $\pi(y) \not\in \pi(X)$.
	Consider the case in which $\pi(y) \in \pi(X)$.
	We have $\dim \pi(X)<n-1$ by \cite[Proposition 2.6(6)]{FKK}.
	Any open box containing the point $\pi(y)$ is not contained in $\pi(X)$.
	In other words, $\pi(y)$ belongs to $\partial (\pi(U) \setminus \pi(X))$.
	By the curve selection lemma \cite[Corollary 2.9]{Fuji}, there exist $\varepsilon>0$ and a definable continuous map $\delta:(0,\varepsilon) \to \pi(U)$ such that  $\lim_{t \to +0} \delta(t)=\pi(y)$ and $\delta(t) \in \pi(U) \setminus \pi(X)$ for every $0<t<\varepsilon$. 
	Take $0<d<\varepsilon$.
	Consider the definable map $\eta:[0,d] \to U$ given by $\eta(0)=\pi(y)$ and $\eta(t)=\delta(t)$ for $0<t \leq d$.
	The map $\eta$ is continuous.
	We consider the map $\gamma:[0,d] \to U$ given by $\gamma(t)=(\eta(t),y_n)$.
	Note that $\gamma(t) \notin X$ for $0<t \leq d$ because $\eta(t) \notin \pi(X)$.
	The point $\gamma(d)$ is definably reachable from $y$ in $U \setminus X$.
	We may assume that $\pi(y) \notin \pi(X)$ by considering $\gamma(d)$ instead of $y$.
	
	Since $\pi(y) \notin \pi(X)$, the segment connecting $y$ with $z:=(y_1,\ldots, y_{n-1},x_n) \in U$ does not intersect with $X$.
	The point $(y_1,\ldots,y_{n-1})$ is reachable from $(x_1,\ldots,x_{n-1})$ in $U_{x_n} \setminus X_{x_n}$ by the induction hypothesis.
	This implies that $z$ is definably reachable from $x$ in $U \setminus X$.
	We have shown that $y$ is reachable from $x$ in $U \setminus X$.
\end{proof}

\begin{remark}\label{rem:connected}
	L. van den Dries proposed properties which are expected to be possessed by a `tame' dimension function in \cite[Definition]{vdD-dim}.
	Proposition \ref{prop:connected}  holds under the assumption that the dimension function $\dim$ enjoys these properties without assuming local o-minimality when the structure is a definably complete expansion of an ordered field.
\end{remark}
\begin{proof}
	In the proof of Proposition \ref{prop:connected}, we reduce to the case in which $\pi(y) \notin \pi(X)$.
	This part of the proof should be revised, but revision is not required for the other part of the proof because the properties of dimension function found in \cite[Section 1]{vdD-dim} are only used in the remaining part.
	
	We reduce to the case in which $\pi(y) \notin \pi(X)$.
	We consider the unit sphere $S^{n-1}=\{u=(u_1,\ldots,u_n) \in M^n\;|\; \sum_{i=1}^n u_i^2=1\}$.
	Observe that $\dim S^{n-1}=n-1$.
	Define the definable map $\kappa: X \to S^{n-1}$ by $\kappa(u)=(u-y)/\|u-y\|$, where $\|\cdot \|$ denotes the Euclidean norm in $M^n$.
	The map $\kappa$ is well-defined because $y \notin X$.
	We have $\dim \kappa(X) \leq n-2$ by \cite[Corollary 1.5(ii)]{vdD-dim}.
	In particular, we can find $v \in S^{n-1} \setminus \kappa(X)$.
	We can reduce to the case $\pi(y) \notin \pi(X)$ by a linear change of coordinate so that the vector $v$ coincides with the unit vector $(0,\ldots,0,1)$.
\end{proof}

We use the following corollary in the proof of Theorem \ref{thm:open_map}.
\begin{corollary}\label{cor:separate}
	%Consider a definably complete locally o-minimal expansion of an ordered group $\mathcal M=(M,<,\ldots)$.
	Let $B$ be an open box in $M^n$ and $W$ be a definable open subset of $B$ such that $B \not\subseteq \mycl(W)$.
	Then the equality $\dim (B \cap \partial W)=n-1$ holds.
\end{corollary}
\begin{proof}
	Assume for contradiction that $\dim (B \cap \partial W)<n-1$.
	Observe that $B \cap \partial W$ is closed in $B$.
	The definable set $B \setminus \partial W$ is decomposed into two disjoint definable open sets $W$ and $B \setminus \mycl(W)$.
	This contradicts Proposition \ref{prop:connected}. 
\end{proof}

We are now ready to prove Theorem \ref{thm:open_map}.
%We do not use the notion of homology like the proof of \cite{Johns}, but we use a feature of local o-minimality proven in \cite{Fuji2}.

\begin{proof}[Proof of Theorem \ref{thm:open_map}]
	Let $U$ be an open subset of $V$.
	We may assume that $U$ is a bounded open box such that $\mycl(U) \subseteq V$ without loss of generality because any open subset of $V$ is a union of possibly infinitely many such open boxes.
	Set $$H=f(U) \cap \mybd(f(U))$$ as in the proof of \cite{Johns}.
	Observe that $\dim H \leq n-1$ because $H$ has an empty interior. 
	We have only to show that $H=\emptyset$.
	Assume for contradiction that $H \neq \emptyset$.
	
	We first show that $H$ is an open subset of $\mybd(f(U))$.
	We have $f(\mycl(U)) \subseteq \mycl(f(U))$ because $f$ is continuous.
	We prove the opposite inclusion.
	$f(\mycl(U))$ is closed and bounded by \cite[Proposition 1.10]{Miller} because $\mycl(U)$ is closed and bounded.
	We have $\mycl(f(U)) \subseteq f(\mycl(U))$ because $f(\mycl(U))$ is closed.
	We have proven that $f(\mycl(U)) = \mycl( f(U))$.
	We get $\partial f(U) =\mycl(f(U)) \setminus f(U)=  f(\mycl(U)) \setminus f(U)=f(\partial U)$ because $f$ is injective.
	Observe that $f(\partial U)$ is closed and bounded by \cite[Proposition 1.10]{Miller} for the same reason.
	This implies that $\partial f(U)$ is closed.
	The set $H$ is open in $\mybd(f(U))$ because of the equality $H= \mybd(f(U)) \setminus \partial f(U)$.
	
	We show that $\dim H=n-1$.
	Fix a point $b \in H$.
	Since $H$ is open in $\mybd(f(U))$, we can take an open box $B$ such that $b \in B$ and $B \cap \mybd(f(U)) \subseteq H$.
	Set $W:=B \cap \myint(f(U))$ and we prove that $W$ is not an empty set.
	By the definition of $\mybd(f(U))$, $B$ has a nonempty intersection with $f(U)$ because $b \in \mybd(f(U))$.
	$f^{-1}(B) \cap U$ is a nonempty definable open set because $f$ is continuous.
	In particular, we have $\dim (f^{-1}(B) \cap U)=n$.
	We have $\dim (B \cap f(U))=\dim f(f^{-1}(B) \cap U))=n$ by \cite[Proposition 2.8(6)]{FKK} because $f$ is injective.
	This means that $B \cap f(U)$ has a nonempty interior.
	We have proven that $W=B \cap \myint(f(U)) \neq \emptyset$.
	
	Assume for contradiction that $B \subseteq \mycl(W)$.
	We have $B \subseteq \mycl(W) \subseteq \mycl(f(U))$ by the definition of $W$.
	We get $B \subseteq f(U)$ because $B \cap \mybd(f(U)) \subseteq H \subseteq f(U)$.
	In particular, the point $b$ belongs to $\myint(f(U))$, which contradicts to the definition of $H$ and the assumption $b \in H$.
	We have shown that $B \not\subseteq \mycl(W)$.
	We get $\dim (B \cap \partial W)=n-1$ by applying Corollary \ref{cor:separate} to $W$.
	Observe that $B \cap \partial W \subseteq B \cap \mybd(f(U)) \subseteq H$.
	Recall that $\dim H \leq n-1$.
	The equality $\dim H=n-1$ follows from these two facts.
	
	Set $D=f^{-1}(H)$.
	Observe that $D \subseteq U$ and $\dim D=n-1$ by \cite[Propositon 2.8(6)]{FKK} because $f$ is injective.
	Let $\pi:M^n \to M^{n-1}$ be the coordinate projection forgetting the last coordinate.
	After permuting the coordinates if necessary, we can choose
	\begin{itemize}
		\item a definable open subset $C$ of $M^{n-1}$ contained in $\pi(D)$,
		\item a definable open subset $O$ of $M^n$ and 
		\item a definable continuous map $g:C \rightarrow D$
	\end{itemize}
	such that 
	\begin{itemize}
		\item $\pi(O)=C$, 
		\item $D \cap O = g(C)$ and 
		\item the composition $\pi \circ g$ is the identity map on $C$  
	\end{itemize}
	by \cite[Lemma 3.5]{Fuji2}.
	In \cite[Lemma 3.5]{Fuji2}, the structure $\mathcal M$ is assumed to possess a property called the property (a), but every definably complete locally o-minimal structure possesses the property (a) by \cite[Theorem 2.5]{FKK} and this assumption can be omitted.
	Observe that $C \subseteq \pi(D) \subseteq \pi(U)$.
	We may assume that $O \subseteq U$ by replacing $O$ with $O \cap U$.
	$D \cap O$ is the graph of a definable continuous function $h:C \to M$.
	Set $O_1=\{(x,y) \in (C \times M) \cap O\;|\; y>h(x)\}$ and $O_2=\{(x,y) \in (C \times M) \cap O\;|\; y<h(x)\}$.
	We obviously have $D \cap O \subseteq \partial O_1 \cap \partial O_2$ and $\dim (D \cap O)=n-1$.
	Set $$H'=f(D \cap O)=f(f^{-1}(H) \cap O).$$
	We have $\dim H'=n-1$ by \cite[Propositon 2.8(6)]{FKK} and the injectivity of $f$.
	We get $$H' \subseteq \partial f(O_1) \cap \partial f(O_2)$$ because $f$ is continuous and injective.
	
	We have $\dim \mybd (f(O_i)) \leq n-1$ for $i=1,2$ because $\mybd( f(O_i))$ has an empty interior.
	We have $\dim \mybd( f(O_i))=n-1$ because $H'$ is contained in $\mybd(f(O_i))$.
	We have $\dim(\mybd( f(O_1)) \cup \mybd(f(O_2)))=n-1$ by \cite[Proposition 2.8(5)]{FKK}.
	We apply \cite[Lemma 3.5]{Fuji2} to $H'$ and $ \mybd( f(O_1)) \cup \mybd(f(O_2))$.
	After permuting the coordinates if necessary, we can choose
	\begin{itemize}
		\item a definable open subset $C'$ of $M^{n-1}$ contained in $\pi(H')$,
		\item a definable open subset $O'$ of $M^n$ and 
		\item a definable continuous map $g':C' \rightarrow H'$
	\end{itemize}
	such that 
	\begin{itemize}
		\item $\pi(O')=C'$, 
		\item $(\mybd( f(O_1)) \cup \mybd(f(O_2))) \cap O' = g'(C')$ and 
		\item the composition $\pi \circ g'$ is the identity map on $C'$.
	\end{itemize}
	%Recall that $\pi:M^n \to M^{n-1}$ denotes the coordinate projection forgetting the last coordinate.
	The equality $(\mybd( f(O_1)) \cup \mybd(f(O_2))) \cap O' = H' \cap O'$ is obvious because of the relations $(\mybd( f(O_1)) \cup \mybd(f(O_2))) \cap O' = g'(C')$, $g'(C') \subseteq H'$ and $H' \subseteq \mybd( f(O_1)) \cup \mybd(f(O_2))$.
	In addition, we have 
	\begin{align*}
		H' \cap O'= (\partial (f(O_1)) \cap \partial(f(O_2))) \cap O' = (\mybd(f(O_1)) \cup \mybd(f(O_2))) \cap O'
	\end{align*}
	because $H' \subseteq \partial (f(O_1)) \cap \partial(f(O_2)) \subseteq \mybd(f(O_1)) \cup \mybd(f(O_2))$.
	Using these equalities, we get $$H' \cap O'=\partial(f(O_i)) \cap O'=\mybd(f(O_i)) \cap O'$$ for $i=1,2$.
	There exists a definable continuous function $h':C' \to M$ such that $H' \cap O'$ is the graph of $h'$.
	Set $O'_1=\{(x,y) \in (C' \times M) \cap O'\;|\; y>h'(x)\}$ and $O'_2=\{(x,y) \in (C' \times M) \cap O'\;|\; y<h'(x)\}$.
	Observe that $O' \setminus H'=O'_1 \cup O'_2$.
	By the definable choice lemma \cite[Lemma 2.8]{Fuji}, we can find definable functions $h_i:C' \to M$ such that $(x,h_i(x)) \in O_i$ for each $x \in C'$ and $i=1,2$.
	Using \cite[Proposition 2.8(5),(7)]{FKK}, we may assume that $C'$ is an open box and $h_i$ are continuous on $C'$ for $i=1,2$ by taking a small open box contained in $C'$ in stead of $C'$ if necessary.
	We may assume that $O'$ is of the form $\{(x,y) \in C \times M\;|\; h_2(x)<y<h_1(x)\}$ by replacing $O'$ with $\{(x,y) \in C \times M\;|\; h_2(x)<y<h_1(x)\}$.
	In particular, the definable sets $O'_1$ and $O'_2$ are definably connected.
	We have succeeded in reducing to the cases in which $O'_1$ and $O'_2$ are definably connected.
	
	The open set $O'\setminus H'$ intersects with $f(O_1)$ because $H' \subseteq \partial (f(O_1))$ and $O'$ is a neighborhood of $H' \cap O'$.
	We get $O' \cap (f(O_1) \setminus \myint(f(O_1))) \subseteq H'$ because $\mybd(f(O_1)) \cap O'=H' \cap O'$.
	Therefore, at least one of $O'_1$ and $O'_2$ intersects with $\myint(f(O_1))$ because $O' \setminus H'=O'_1 \cup O'_2$.
	Assume that $O'_1$ intersects with $\myint(f(O_1))$ without loss of generality.
	We can treat the remaining case similarly.
	We have $O'_1 \cap \mybd(f(O_1))=\emptyset$ because $\mybd(f(O_1)) \cap O'=H' \cap O'$.
	We have $O'_1 \subseteq f(O_1)$.
	Otherwise, the definably connected definable set $O'_1$ is decomposed into  two disjoint nonempty definable open sets $O'_1 \cap \myint(f(O_1))$ and $O'_1 \setminus \mycl(f(O_1))$, which is absurd.
	
	At least one of $ O'_1$ and $O'_2$ is contained in $f(O_2)$ for the same reason.
	We have $f(O_1) \cap f(O_2)= \emptyset$ because $O_1 \cap O_2=\emptyset$ and $f$ is injective.
	In addition, $f(O_2) \cap O' \neq \emptyset$ because $\partial f(O_2) \cap O'=H' \cap O' \neq \emptyset$ and $O'$ is open.
	These imply that the definable set $O'_2$ is contained in $f(O_2)$.  
	In summary, we have shown that $O'$ is contained in $f(U)$.
	Take $c' \in C'$ and set $b'=g'(c')$.
	We have $b' \in \myint(f(U))$ because $b'$ is contained in $O'$.
	On the other hand, we have $b' \in H'=f(f^{-1}(H) \cap O) \subseteq H \subseteq \mybd(f(U))$, which is absurd.
	We have finally proven that $H$ is empty.
	This means that $f(U)$ is open; and consequently, $f$ is an open map.
\end{proof}

In Theorem \ref{thm:open_map}, we assume that the domain of $f$ is a definable open set and the codomain is $M^n$.
We can easily extend our main theorem to the definable maps between locally definable manifolds.
%We first recall the definition of locally definable manifolds.
A \textit{locally definable manifold} of dimension $n$ is a triple $(S,(U_i,\theta_i)_{i \leq \kappa})$ where:
\begin{itemize}
		\item $S$ is a set and $(U_i)_{i \leq \kappa}$ is a  cover of $S$;
		\item each $\theta_i:U_i \to M^n$ is an injection such that $\theta(U_i)$ is a definable open subset of $M^n$;
		\item for all $i,j$, $\theta_i(U_i\cap U_j)$ is a definable open subset of $\theta_i(U_i)$ and the transition maps $\theta_{ij}:\theta_i(U_i\cap U_j) \to \theta_j(U_i\cap U_j)$ given by $\theta_{ij}(x)=\theta_j(\theta_i^{-1}(x))$ are definable homeomorphisms.
\end{itemize}
We call the $(U_i,\theta_i)$'s the \textit{definable atlas} of $S$.
We say that $S$ is a locally definable manifold by abuse of terminology if there exists the family satisfying the above conditions $(U_i,\theta_i)_{i \leq \kappa}$.
A locally definable manifold $S$ is equipped with the topology so that every $\theta_i$ is a homeomorphism onto its image.  

We recall another notion.
Let $f:X \to S$ be a map from a topological space $X$ to a set $S$.
We say that it is \textit{locally injective} if, for any $x \in X$, there exists an open neighborhood $U$ of $x$ in $X$ such that the restriction of $f$ to $U$ is injective. 

We  generalize Theorem \ref{thm:open_map} as follows:
\begin{corollary}\label{cor:open_map1}
%	Let $\mathcal M$ be as in Theorem \ref{thm:open_map}.
	Let $X$ and $Y$ be locally definable manifolds of dimension $d$ and $f:X \to Y$ be a definable, continuous and locally injective map.
	Then, $f$ is an open map.
\end{corollary}
\begin{proof}
	Let $U$ be an open subset of $X$.
	Take an arbitrary point $y_0 \in f(U)$.
	We have only to show that there exists an open neighborhood $V$ of $y_0$ in $Y$  contained in $f(U)$.
	We choose a point $x_0 \in X$ such that $f(x_0)=y_0$.
	
	We first choose a definable open neighborhood $W_1$ of $y_0$ in $Y$, a definable open neighborhood $W_2$ of the origin in $M^d$  and a definable homeomorphism $\psi:W_1 \to W_2$ such that $\psi(y_0)$ is the origin.
	It is possible because $Y$ is a locally definable manifold.
	For the same reason, we can take a definable open neighborhood $N_1$ of $x_0$ in $X$, a definable open neighborhood $N_2$ of the origin in $M^d$  and a definable homeomorphism $\varphi:N_1 \to N_2$ such that $\varphi(x_0)$ is the origin.
	We may assume that $N_1 \subseteq U \cap f^{-1}(W_1)$ by shrinking $N_1$ if necessary.
	Furthermore, we may assume that the restriction of $f$ to $N_1$ is injective by shrinking $N_1$ again because $f$ is locally injective.
	Set $V=f(N_1)$.
	We show that $V$ is a desired neighborhood.
	
	The definable map $\psi \circ f \circ \varphi^{-1}:N_2 \to M^d$ is continuous and injective.
	It is an open map by Theorem \ref{thm:open_map}.
	This implies that $\psi(V)=\psi(f(N_1))=\psi(f(\varphi^{-1}(N_2)))$ is open in $M^d$.
	Since $W_2$ is open in $M^d$, the set $\psi(V)$ is open in $W_2$.
	The set $V=\psi^{-1}(\psi(V))$ is contained and open in $W_1$, and it is also open in $Y$.
	The open set $V$ is an open neighborhood of $y_0$ in $Y$ contained in $f(U)$.
\end{proof}

\end{document}